\documentclass[]{article}

\usepackage{mathdots, hyperref}
\usepackage{amsmath, amsthm, amssymb, color, float}
\newtheorem{thm}{Theorem}

\newtheorem{lem}{Lemma}
\usepackage{caption}
\usepackage{subcaption}

\newcommand{\case}{\textbf}

\usepackage{tikz}
\usepgflibrary{arrows}
\usetikzlibrary{calc}
\tikzset{vertex/.style={draw, fill, inner sep = .05cm, circle}}
\tikzset{Lvertex/.style={draw, semithick, inner sep = .05cm, circle}}
\tikzset{edge/.style={}}
\tikzset{Ledge/.style={dashed}}
\tikzset{part1/.style={red}}
\tikzset{part2/.style={cyan}}

\usepackage[final]{showlabels}

\renewcommand{\ss}{\subseteq} 
\newcommand{\nss}{\not\subseteq}

\newcommand{\sizeof}[1]{\left| #1 \right|}

\newcommand{\iso}{\cong}

\newenvironment{ctikz}{\begin{center}\begin{tikzpicture}}{\end{tikzpicture}\end{center}}

\begin{document}

\bibliographystyle{amsplain}

\title{Characterization of Planar, 4-Connected, $K_{2,5}$-minor-free Graphs}

\author{Zach Gaslowitz\footnote{Department of Mathematics, Vanderbilt University, Nashville, Tennessee, U.S.A. (\textit{j.zachary.gaslowitz@vanderbilt.edu})} \qquad Emily A. Marshall\footnote{Department of Mathematics, Louisiana State University, Baton Rouge, Louisiana, U.S.A. (\textit{emarshall@lsu.edu})} \qquad Liana Yepremyan\footnote{School of Computer Science, McGill University, Montreal, Quebec, Canada (\textit{liana.yepremyan@mail.mcgill.ca})}}

\date{\today}

\maketitle

\begin{abstract}
	We show that every planar, $4$-connected, $K_{2,5}$-minor-free
graph is the square of a cycle of even length at least six.
\end{abstract}

\section{Introduction}

All graphs in this paper are finite and simple. A graph is a \emph{minor} of another graph if the first
can be obtained from a subgraph of the second by contracting edges and deleting any resulting
loops and parallel edges. We say that a graph \(G\) is \(H\)-\emph{minor-free} if \(H\) is not a minor of \(G\).

The most well-known result concerning characterizations of minor-free graphs is Wagner's demonstration in \cite{Wagner1937} that $K_5$ and $K_{3,3}$-minor-free graphs are precisely the planar graphs. 
A related result that follows from a different formulation of Wagner's theorem is that a $2$-connected graph is $K_{2,3}$-minor-free if and only if it is $K_4$ or outerplanar.
Other important results in this area include Dirac's \cite{Dirac1952} characterization of all $K_4$-minor-free graphs and more recently, Ding and Liu's \cite{Ding2013} description of $H$-minor-free graphs for all $3$-connected graphs $H$ on at most eleven edges.
In \cite{Ellingham2014}, Ellingham et. al.\ provide a complete characterization of all $K_{2,4}$-minor-free graphs. These types of questions have seen more attention since the conclusion of Robertson and Seymour's Graph Minors Project, which proved that all minor-closed families of graphs can be characterized by a finite set of forbidden minors. 

In this paper we focus on $K_{2,5}$-minor-free graphs. We suspect that this family is large and rather complex so we restrict our attention here to $4$-connected planar $K_{2,5}$-minor-free graphs. 
We choose $4$-connected because the characterization of $5$-connected $K_{2,5}$-minor-free graphs is easy: every $5$-connected graph either has $5$ disjoint paths between a non-adjacent pair of vertices (by Menger's theorem) and hence a $K_{2,5}$ minor, or is $K_6$.

To state our main result we need the following definition. The \emph{square} of a graph $G$, denoted by $G^2$, is a graph on the same vertex set as $G$, with pairs of vertices adjacent in $G^2$ if they are at distance at most two in $G$ (see Fig.~\ref{fig:squaredCycle}).  
\begin{thm} \label{thm:main}
	A graph is planar, $4$-connected and $K_{2,5}$-minor-free if and only if it is the square of a cycle of even length at least $6$.
\end{thm}

We introduce an equivalent definition of minor, which we will use in this paper. $H$ is a \emph{minor} of $G$ if for every vertex $v \in H$, there exists a connected subset of vertices $B_v \ss G$ called the \emph{branch set} of $v$ such that the branch sets of distinct vertices are disjoint and for each edge $vw$ of $H$, there is an edge in $G$ connecting $B_v$ and $B_w$. Note that if $G$ has $K_{2,t}$ or $K_{3,t}$ as a minor, we may assume that the branch sets of vertices in the part of size $t$ each consist of only a single vertex. If $G$ has $K_{2,5}$ with bipartition $(\{v_1, v_2\}, \{w_1, \ldots, w_5\})$ as a minor, where $\sizeof{B_{w_i}} = 1$, we will say that the minor is given by $\{B_{v_1},B_{v_2}; S\}$, where $S = \bigcup_{i=1}^5 B_{w_i}$.
We will analogously describe $K_{3,t}$ minors.

For a given graph $G$ and any vertex $v \in G$, the \emph{open neighborhood} $N(v)$ denotes the set of vertices of $G$ adjacent to $v$.
Similarly, for vertices $v_1, \ldots, v_n \in G$, $N(v_1, \ldots, v_n) = (\bigcup_{i=1}^n N(v_i))\setminus \{v_1, \ldots, v_n\}$.
The \emph{closed neighborhood} is defined to be $N[v] := N(v) \cup \{v\}$ and $N[v_1, \ldots, v_n] := N(v_1, \ldots, v_n) \cup \{v_1, \ldots, v_n\}$.

Given a graph $G$, its \emph{line graph} $L(G)$ is a graph with vertex set \(V(L(G)) = E(G)\) and edge set consisting of pairs of vertices of $L(G)$ whose corresponding edges in $G$ share a common endpoint.

\section{Proof of Theorem~\ref{thm:main}}

The proof of Theorem~\ref{thm:main} uses the following result of  Martinov  from \cite{Martinov1982}.
His result requires the following definition. A \emph{cyclically 4-edge-connected} graph is a $3$-edge-connected graph with no $3$-edge-cuts that leave a cycle in each component. 
\begin{thm}[Martinov, \cite{Martinov1982}] \label{thm:martinov}
	A 4-connected graph that is 4-regular and has every edge in a triangle is either a squared cycle of length at least five or the line graph of a cubic, cyclically 4-edge-connected graph.
\end{thm}

Additionally, our proof uses the following lemmas, each of which describes the structure of planar, $4$-connected, $K_{2,5}$-minor-free graphs.
These lemmas together with Martinov's theorem imply that $4$-connected, $K_{2,5}$-minor-free graphs graphs must be the squares of cycles of length at least five. We then show that only even squared cycles of length at least six are
both planar and 4-connected, thus finishing the proof.
\begin{lem} \label{lem:nbhdCutSet}
	For any vertex $v$ in a 4-connected planar graph $G$, $N[v]$ is not a cut set.
\end{lem}
\begin{proof}
	Suppose, to the contrary, that for some $v\in V(G)$, $N[v]$ is a cut set.
	Let $S \ss N(v)$ be a minimal cut set of the 3-connected graph $G \setminus v$.
	Note that $|S|\geq 3$.
	Let $C_1$ and $C_2$ to be two distinct components of $G \setminus (\{v\} \cup S)$.
	Then $(\{v\}, C_1, C_2; S)$ gives a $K_{3,|S|}$ minor, and in particular a $K_{3,3}$ minor, contradicting planarity.
\end{proof}

\begin{lem} \label{lem:4reg}
	If $G$ is a planar, 4-connected, $K_{2,5}$-minor-free graph, then $G$ is 4-regular.
\end{lem}
\begin{proof}
	To be 4-connected, $G$ must have minimum degree at least four.
	Assume that some vertex $v$ of $G$ has degree $n>4$.
	Fix a planar embedding of $G$ and label $v$'s neighbors $w_1, \ldots, w_n$, ordered clockwise around $v$.
%
	
	Note that $w_i$ cannot be adjacent to any $w_j$ for $j \ne i \pm 1$, taking the indices modulo $n$.
	Otherwise, $\{v, w_i, w_j\}$ would be a 3-cut.
	Because $d(w_i) \ge 4$ for each $i$, we see that $N(w_i) \nss N[v]$.
	
	By Lemma \ref{lem:nbhdCutSet}, $C=G\setminus N[v]$ is connected, so $G$ has a $K_{2,n}$ minor given by $((C, \{v\} ; \{w_1, \ldots, w_n\}))$.
	This contradicts our choice of $G$, so $G$ must be  4-regular.
\end{proof}

\begin{lem} \label{lem:triangles}
	If $G$ is a planar, 4-connected, $K_{2,5}$-minor-free graph, then every edge of $G$ is in a triangle.
\end{lem}
\begin{proof}
	Assume to the contrary that $G$ has an edge $ab$ not part of any triangle.
	We know from Lemma \ref{lem:4reg} that $G$ is 4-regular, so $a$ and $b$ each have three neighbors, all distinct vertices.
	Fix a planar embedding of $G$ and label these vertices as seen in Figure \ref{fig:abNbhd}, ordered as they appear in this embedding.
	\begin{figure}
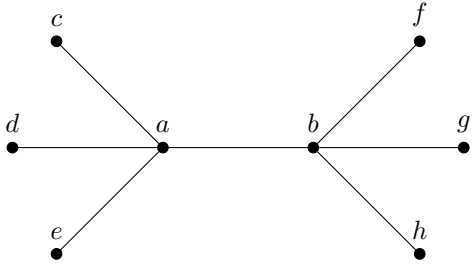

		\begin{ctikz}[scale=2]
			\draw[edge] (0,0)node[vertex, label=above:{$a$}] (a) {}
			   -- ++(135:1)  node[vertex, label=above:{$c$}] (c) {} (a)
			   -- ++(180:1)  node[vertex, label=above:{$d$}] (d) {} (a)
			   -- ++(-135:1) node[vertex, label=above:{$e$}] (e) {} (a)
			   -- ++(1,0)    node[vertex, label=above:{$b$}] (b) {}
			   -- ++(45:1)   node[vertex, label=above:{$f$}] (f) {} (b)
			   -- ++(0:1)    node[vertex, label=above:{$g$}] (g) {} (b)
			   -- ++(-45:1)  node[vertex, label=above:{$h$}] (h) {};
		\end{ctikz}
		\caption{The neighborhood around an edge $ab$ in $G$, not part of a triangle.
		Neighbors are shown as they appear in a planar embedding.}
		\label{fig:abNbhd}
	\end{figure}
	
	Note that $G$ cannot have exactly these eight vertices, because in particular $c$ must have degree four and there are not three other vertices in $N[a,b]$ that $c$ can be adjacent to.
	More specifically, $c \not\sim b$ because $ab$ is not in a triangle, and $c \sim e, g, h$ would yield the 3-cuts $\{a, c, e\}$, $\{b, c, g\}$, and $\{b, c, h\}$ respectively, while $G$ is assumed to be 4-connected.
	So $|V(G)| \ge 9$.
	
	Any component of $G \setminus N[a,b]$ must be adjacent to exactly four vertices in $N(a,b)$ in order for $G$ to be 4-connected without introducing a $K_{2,5}$ minor.
	
	\case{Case 1}: $G \setminus N[a,b]$ has only one component, $C$.\\
	Then $C$ is nonadjacent to exactly two vertices $x,y \in N(a,b)$, which is to say $N(x),N(y) \ss N[a,b]$.
	If $x$ and $y$ have a common neighbor $z \in N(a,b)$, then $G$ has a $K_{2,5}$ minor given by $(C \cup \{z\}, \{a,b\}; N(a,b) \setminus z)$.
	
	If $x$ and $y$ do not have a common neighbor, we must have $x \sim y$ and $N(a,b) \ss N[x,y]$ (the two additional neighbors from $x$ and $y$, all distinct, will completely cover the four remaining vertices of $N(a,b)$).
	Then $(C, \{a,b\}, \{x,y\}; N(a,b) \setminus \{x,y\})$ gives a $K_{3,4}$ minor, contradicting the planarity of $G$.
	
	\case{Case 2}: $G \setminus N[a,b]$ has more than one component.\\
	Take any two of them, $C_1$ and $C_2$.
	If $C_1$ and $C_2$ together are adjacent to every vertex of $N(a,b)$, then let $x \in N(a,b)$ be one of the two vertices adjacent to both.
	Then $G$ has a $K_{2,5}$ minor given by $(\{a,b\}, \{x\} \cup C_1 \cup C_2; N(a,b) \setminus x)$.
	Otherwise, $C_1$ and $C_2$ must have at least three common neighbors, so let $S \ss N(a,b)$ consist of any three of them.
	Then $G$ has a $K_{3,3}$ minor given by $(C_1, C_2, \{a,b\}; S)$, contradicting planarity.
	
	Either way, we see that every edge of $G$ must be in a triangle.
\end{proof}

Note that the next lemma does not assume planarity.
\begin{lem} \label{lem:notLineGraph}
	The line graph of any cubic, 3-connected graph $H$ has $K_{2,5}$ as a minor, unless $H\iso K_4$.
\end{lem}
\begin{proof}
	Consider any cubic, 3-connected graph $H$ not isomorphic to $K_4$.
	Note that any such graph must have some edge, $uv$, not in a triangle.
	Let $w$, $x$ and $y$,  $z$ be the (necessarily distinct) neighbors of  $u$ and $v$, respectively.
	Although the two neighbors of $w$ other than $v$ may not be distinct from $x$, $y$, and $z$, call them $s$ and $t$, as in Figure \ref{fig:uvDecoratedNbhd}.
	\begin{figure}[h]
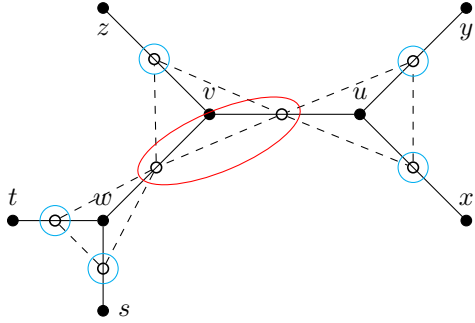

%
		\begin{ctikz}[scale=2, rotate=-90, xscale=-1]
			\draw[edge] (-1,1) node[vertex, label=above:{$x$}] (x) {}
			   -- node[Lvertex] (ux) {}
			      ++(-45:1) node[vertex, label=above:{$u$}] (u) {}
			   -- node[Lvertex] (uy) {}
			      ++(45:1) node[vertex, label=below:{$y$}] (y) {}
			   (u)
			   -- node[Lvertex] (uv) {}
			      ++(0,-1) node[vertex, label=above:{$v$}](v) {}
			   -- node[Lvertex] (vw) {}
			      ++(-135:1) node[vertex, label=above:{$w$}] (w) {}
			   (v)
			   -- node[Lvertex] (vz) {}
			      ++ (-45:1) node[vertex, label=below:{$z$}] (z) {};
			
			\draw[edge] (w)
			   -- node[Lvertex] (wt) {}
			      ++ (-90:.6) node[vertex, label=above:$t$] (t) {};
			\draw[edge] (w)
			   -- node[Lvertex] (ws) {}
			      ++ (180:.6) node[vertex, label=right:$s$] (s) {};
			
			\draw[Ledge] (uv) -- (ux) -- (uy) -- (uv) -- (vw) -- (ws) -- (wt) -- (vw) -- (vz) -- (uv);
			
			\draw[rotate=67, part1] ($(uv)!0.5!(vw)$) ellipse (.58 and .2);
			\draw[part2] (ux) circle (.1);
			\draw[part2] (uy) circle (.1);
			\draw[part2] (ws) circle (.1);
			\draw[part2] (wt) circle (.1);
			\draw[part2] (vz) circle (.1);
		\end{ctikz}
		\caption{A neighborhood around $uv$ in $H$, along with the corrosponding portion of $L(H)$, highlighting the $K_{2,5}$ minor, with the branch set of the remaining vertex given by the entire rest of the line graph.}
		\label{fig:uvDecoratedNbhd}
	\end{figure}
	
	Because $H$ is 3-connected, $H \setminus \{u,w\}$ is connected.
	Note that $v$ cannot be a cut vertex of $H \setminus \{u,w\}$, so $H \setminus \{u,v,w\}$ is connected.
	Because in particular $x \ne y$, $H \setminus \{u,v,w\}$ must contain an edge, so it will induce a connected subgraph of $L(H)$ which avoids the edges $uv$, $ux$, $uy$, $vz$, $vw$, $ws$, and $wt$, which are all distinct.
	Then $L(H)$ has a $K_{2,5}$ minor given by $(\{uv, vw\}, L(H) \setminus N[uv,vw]; N(uv, vw))$, where neighborhoods are taken in $L(H)$.
	See Figure \ref{fig:uvDecoratedNbhd}.
\end{proof}

\begin{proof}[Proof of Theorem \ref{thm:main}]
	Let $G$ be any planar, 4-connected, $K_{2,5}$-minor-free graph.
	By Theorem \ref{thm:martinov}, along with Lemmas \ref{lem:4reg} and \ref{lem:triangles}, we see that $G$ must be a squared cycle or the line graph of a cubic, cyclically 4-edge-connected graph.
	Noting that $L(K_4) \iso C_6^2$ and that a cubic, cyclically 4-edge-connected graph is, in particular, 3-connected, Lemma \ref{lem:notLineGraph} ensures that $G$ is a squared cycle of length at least five.

	
	Now note that $C_n^2$ will have $C_{n-2}^2$ as a minor for all odd $n \ge 5$, so in particular will have $C_5^2 \iso K_5$ as a minor.
	Thus, $C_n^2$ is nonplanar for all such $n$, completing the forward direction of Theorem \ref{thm:main}.
	
	For the reverse direction, fix an even $n \ge 6$ and consider $C_n^2$.
	Note that $C_n^2$ can be embedded in the plane with two vertex disjoint faces $F_1$ and $F_2$ of degree $\frac{n}{2}$ connected by $n$ triangular faces.
	See, for example, Figure \ref{fig:squaredCycle}.
	\begin{figure}[h!]
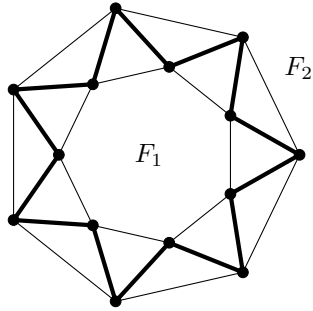

		\begin{ctikz}[scale=2]
			\def\n{7}
			\def\innerR{.6}
			\node at (0,0) {$F_1$};
			\node at (30:1.15) {$F_2$};
			\foreach \i in {1,2,...,\n}
			{
				\draw[edge] (\i/\n*360:1) node[vertex] {}
				         -- (\i/\n*360+1/\n*360:1)
				            (\i/\n*360-1/\n*180:\innerR) node[vertex] {}
				         -- (\i/\n*360+1/\n*180:\innerR);
				\draw[edge, ultra thick] (\i/\n*360-1/\n*180:\innerR)
				         -- (\i/\n*360:1)
				         -- (\i/\n*360+1/\n*180:\innerR);
			}
		\end{ctikz}
		\caption{An embedding of $C_{14}^2$, with the original 14-cycle shown in bold.}
		\label{fig:squaredCycle}
	\end{figure}
	
	Note that $C_n^2$ must be 4-connected, because any cut set must contain at least two vertices from each of $F_1$ and $F_2$.
	
	Suppose, now, that $C_{n}^2$ has a $K_{2,5}$ minor given by $(R_1, R_2; S)$.
	Then without loss of generality, $F_1$ contains three vertices $S' \ss S$.
	Form a new graph $G'$ from $G$ by adding a new vertex $v$ adjacent to these three vertices.
	Then $G'$ is planar, but has a $K_{3,3}$ minor given by $(R_1, R_2, \{v\}; S')$, a contradiction. Thus, $C_n^2$ is planar, 4-connected, and $K_{2,5}$-minor-free, for all even $n \ge 6$, completing the proof of Theorem \ref{thm:main}.
\end{proof}
\section{Acknowledgements}
The authors would like to thank Mark Ellingham for his insight and helpful suggestions while they worked on this problem.

\bibliography{k2t}

\end{document}